\documentclass[11pt]{article}

\usepackage[T1]{fontenc}
\usepackage{amsmath,amssymb,amsthm,mathtools}
\usepackage[margin=1in]{geometry}
\usepackage{booktabs}
\usepackage[colorlinks=true,linkcolor=blue,citecolor=blue,urlcolor=blue]{hyperref}

\allowdisplaybreaks

\newtheorem{theorem}{Theorem}[section]
\newtheorem{proposition}[theorem]{Proposition}
\newtheorem{lemma}[theorem]{Lemma}
\newtheorem{corollary}[theorem]{Corollary}
\newtheorem{remark}[theorem]{Remark}
\theoremstyle{definition}
\newtheorem{definition}[theorem]{Definition}

\DeclareMathOperator{\diss}{diss}

\DeclareMathOperator{\Row}{Row}
\DeclareMathOperator{\Tr}{Tr}
\newcommand{\F}{\mathbb{F}}
\newcommand{\Mat}{\mathrm{M}}
\newcommand{\GL}{\mathrm{GL}}
\newcommand{\cS}{\mathcal{S}}
\newcommand{\ind}{\mathbf{1}}
\newcommand{\qbinom}[2]{\genfrac{[}{]}{0pt}{}{#1}{#2}_{q}}

\title{Laplacian Bounds for the Dissociation Number of\\
Regular Graphs of Matrix Rings}
\author{Joyentanuj Das\footnote{Department of Mathematics, College of Engineering and Technology, SRM Institute of Science and Technology, Kattankulathur, Chennai 603203, India. Emails: joyentanuj@gmail.com, joyentad@srmist.edu.in.}}
\date{}

\begin{document}

\maketitle

\begin{abstract}
Let $\Gamma_n(q)$ be the graph whose vertices are the invertible matrices in $\Mat_n(\F_q)$, with two distinct matrices adjacent whenever their sum is singular.  A dissociation set is a vertex set inducing a graph of maximum degree at most one.  We study the dissociation number of $\Gamma_n(q)$ by embedding it as an induced subgraph of the total graph $T_n(q)$ on all of $\Mat_n(\F_q)$.  A general Laplacian inequality for $k$-independent sets, together with an explicit character computation for the additive group of the matrix ring, gives parity-sensitive upper bounds.  For fixed $n$, the resulting bound is of order at most $q^{n^2-n+1}$ for odd $q$ and at most $q^{n^2-2n+2}$ for even $q$.  In particular,
\[
  \diss(\Gamma_n(q))\le q^{n^2-n+1}-1.
\]
In the other direction, the regular representation of the extension field
$\F_{q^n}$ gives $\diss(\Gamma_n(q))\ge q^n-1$.  We give complete proofs,
including a self-contained derivation of the required matrix character sum,
and determine the smallest case: $\diss(\Gamma_2(2))=3$.
\end{abstract}



\section{Introduction}

Graphs attached to rings provide a meeting point for algebra, combinatorics,
and spectral graph theory.  Among the best-known examples are zero-divisor
graphs, unitary Cayley graphs, total graphs, and regular graphs.  Anderson and
Badawi \cite{anderson-badawi} introduced the total graph of a commutative
ring: its vertices are all ring elements, and distinct elements are adjacent
when their sum is a zero-divisor.  Dol\v{z}an and Oblak
\cite{dolzan-oblak} extended this construction to finite, possibly
noncommutative rings.  A broad account of graphs defined from ring-theoretic
data can be found in \cite{anderson-book}.

The present paper concerns the full matrix ring
$\Mat_n(\F_q)$ over the finite field $\F_q$.  Its zero-divisors are exactly
the singular matrices, while its regular elements are exactly the
invertible matrices.  The associated regular graph, denoted here by
$\Gamma_n(q)$, has vertex set $\GL_n(\F_q)$; two distinct vertices $A$ and
$B$ are adjacent when $A+B$ is singular.  The word ``regular'' in this name
comes from regular elements of a ring, although $\Gamma_n(q)$ also happens
to be regular in the graph-theoretic sense.

Several classical graph parameters of matrix-ring regular graphs have been
investigated.  Akbari, Jamaali, and Seyed Fakhari
\cite{akbari-clique} proved a bound on their clique number that is independent
of the field size when $q$ is odd.  Akbari, Aryapoor, and Jamaali
\cite{akbari-subgraphs} considered clique and chromatic numbers of related
subgraphs.  Tomon \cite{tomon} established lower and upper bounds for the
chromatic number; in particular, for odd $q$ he obtained
\[
  \chi(\Gamma_n(q))\ge (q/4)^{\lfloor n/2\rfloor}.
\]
Bardestani and Mallahi-Karai \cite{bardestani} used algebraic and
representation-theoretic methods to study chromatic numbers of structured
Cayley graphs related to these matrix graphs.

Nica \cite{nica} shifted attention to the independence number and proved
\begin{equation}\label{eq:nica-bound-intro}
  \alpha(\Gamma_n(q))\le q^{n^2-n+1}.
\end{equation}
The key idea is to pass from $\Gamma_n(q)$ to the total graph $T_n(q)$ on
all matrices.  Although $T_n(q)$ need not be regular when $q$ is odd, its
additive group is abelian, and its Laplacian spectrum can be computed by
finite-field characters.  This avoids the much harder character theory of
$\GL_n(\F_q)$.

Our parameter is the dissociation number.  A dissociation set in a graph is
a vertex set inducing maximum degree at most one; its maximum cardinality is
denoted by $\diss(G)$.  Equivalently, the induced graph is a disjoint union
of isolated vertices and independent edges.  This is the special case
$k=1$ of the $k$-independence number $\alpha_k(G)$ studied, for example, by
Caro and Hansberg \cite{caro-hansberg} and Kogan \cite{kogan}.

The dissociation problem first appeared in the vertex-deletion framework of
Yannakakis \cite{yannakakis}, who showed that its decision version is hard
even on bipartite graphs.  Boliac, Cameron, and Lozin
\cite{boliac-cameron-lozin} strengthened and refined the complexity results
for bipartite graph classes.  More recently, Bock, Pardey, Penso, and
Rautenbach studied relations among dissociation, independence, and induced
matchings \cite{bock-relating}, improvements of the elementary inequality
$\diss(G)\le2\alpha(G)$ on structured graph classes
\cite{bock-independence}, and lower bounds in terms of elementary graph
parameters \cite{bock-bound}.  These works are predominantly structural or
algorithmic; our purpose is to obtain spectral upper bounds for a dense
algebraically defined family.

The inequality $\diss(G)\le2\alpha(G)$ and
\eqref{eq:nica-bound-intro} immediately give
\[
  \diss(\Gamma_n(q))\le2q^{n^2-n+1}.
\]
The central observation of this paper is that one should apply the
Laplacian method directly to a dissociation set instead.  Such a set of
order $d$ spans at most $d/2$ edges.  In the Laplacian cut calculation this
replaces the minimum degree $\delta$ by $\delta-1$, and it removes the
factor $2$ from the preceding estimate.  Retaining the exact largest
Laplacian eigenvalue also reveals a stronger exponent in even
characteristic.

\paragraph{Contributions.}
The main contributions are the following.
\begin{enumerate}
\item We prove the general inequality
\[
  \alpha_k(G)\le |V(G)|
  \left(1-\frac{\delta(G)-k}{\lambda_{\max}(G)}\right),
\]
where $\lambda_{\max}(G)$ is the largest Laplacian eigenvalue.
\item We give a detailed character-theoretic computation of
$\lambda_{\max}(T_n(q))$.  The largest eigenvalue is produced by a rank-one
character for odd $q$, but by a rank-two character for even $q$.
\item We obtain explicit upper bounds for $\diss(T_n(q))$ and hence for
$\diss(\Gamma_n(q))$, together with the uniform estimate
$q^{n^2-n+1}-1$ and parity-sensitive asymptotics.
\item We construct a dissociation set of order $q^n-1$ in
$\Gamma_n(q)$ from the embedding $\F_{q^n}\hookrightarrow\Mat_n(\F_q)$,
and we determine $\diss(\Gamma_2(2))$ exactly.
\end{enumerate}

\paragraph{Organization.}
The paper is organized into five sections.  Section~\ref{sec:notation}
introduces the notation and states the main results.  Section~\ref{sec:lower}
gives the extension-field construction for the lower bound.
Section~\ref{sec:spectral} develops the Laplacian method, evaluates the
required character sums, and proves the upper bounds.  Finally,
Section~\ref{sec:examples} presents small examples, determines the exact
value for $\Gamma_2(2)$, and summarizes the conclusions of the paper.

\section{Notations, preliminary facts, and main results}\label{sec:notation}

All graphs in this paper are finite, simple, and undirected.  For a graph
$G$, its vertex and edge sets are $V(G)$ and $E(G)$, and its order is
$v(G)=|V(G)|$.  If $D\subseteq V(G)$, then $G[D]$ denotes the subgraph
induced by $D$, and
\[
  e_G(D)=|E(G[D])|.
\]
If $D$ and $C$ are disjoint vertex sets, $e_G(D,C)$ is the number of edges
with one endpoint in $D$ and the other in $C$.  The degree of a vertex $u$
is $d_G(u)$, and the minimum and maximum degrees are $\delta(G)$ and
$\Delta(G)$.

Let $A_G$ be the adjacency matrix of $G$, let
$D_G=\operatorname{diag}(d_G(u):u\in V(G))$ be the diagonal degree matrix,
and let
\[
  L_G=D_G-A_G
\]
be the combinatorial Laplacian.  Its eigenvalues are real and nonnegative.
We write $\lambda_{\max}(G)$ for its largest eigenvalue and $\ind$ for the
all-ones vector.  The basic identities
\begin{equation}\label{eq:laplacian-basic}
  L_G\ind=0,
  \qquad
  x^{\mathsf T}L_Gx
  =\sum_{\{u,w\}\in E(G)}(x_u-x_w)^2
\end{equation}
will be used repeatedly.

\begin{definition}
For an integer $k\ge0$, a set $D\subseteq V(G)$ is $k$-independent if
$\Delta(G[D])\le k$.  The maximum cardinality of such a set is the
$k$-independence number $\alpha_k(G)$.  Thus
\[
  \alpha_0(G)=\alpha(G)
  \qquad\text{and}\qquad
  \alpha_1(G)=\diss(G).
\]
\end{definition}

The following elementary comparison explains why dissociation is a natural
relaxation of independence.

\begin{lemma}\label{lem:alpha-diss}
For every graph $G$,
\[
  \alpha(G)\le\diss(G)\le2\alpha(G).
\]
\end{lemma}

\begin{proof}
Every independent set induces maximum degree zero, so it is a dissociation
set.  This gives $\alpha(G)\le\diss(G)$.

For the reverse inequality, let $D$ be a dissociation set.  Every component
of $G[D]$ is either an isolated vertex or a single edge.  Suppose that there
are $s$ isolated vertices and $t$ edge components.  Choose all $s$ isolated
vertices and one endpoint from each of the $t$ edges.  The selected vertices
form an independent set in $G$: vertices from different components of
$G[D]$ have no edges between them, and only one endpoint was selected from
each edge component.  Hence
\[
  \alpha(G)\ge s+t\ge\frac{s+2t}{2}=\frac{|D|}{2}.
\]
Taking $D$ maximum yields $\diss(G)\le2\alpha(G)$.
\end{proof}

We next fix the algebraic notation.  Let $q$ be a prime power, let $\F_q$
be the field with $q$ elements, and let $n\ge2$.  We write
\[
  \Mat_n(q)=\Mat_n(\F_q),
  \qquad
  \GL_n(q)=\GL_n(\F_q).
\]
The set of singular matrices is
\[
  \cS_n(q)=\{X\in\Mat_n(q):\det X=0\}.
\]
Its cardinality will be denoted by
\begin{equation}\label{eq:z-definition}
  z_n(q)=|\cS_n(q)|=q^{n^2}-|\GL_n(q)|.
\end{equation}
For $1\le r\le n$, define
\begin{equation}\label{eq:a-definition}
  a_r(n,q)=q^{n(n-1)/2}\prod_{i=1}^{n-r}(q^i-1),
\end{equation}
where an empty product is $1$.  When no confusion can arise, we abbreviate
$z=z_n(q)$ and $a_r=a_r(n,q)$.

There are two graphs in the paper.
\begin{align*}
 V(\Gamma_n(q))&=\GL_n(q),\\
 V(T_n(q))&=\Mat_n(q),
\end{align*}
and in both graphs distinct matrices $A,B$ are adjacent if and only if
$A+B\in\cS_n(q)$.  Since the adjacency rules agree,
\begin{equation}\label{eq:induced-subgraph}
  \Gamma_n(q)=T_n(q)[\GL_n(q)].
\end{equation}

We record two elementary matrix counts.

\begin{lemma}\label{lem:matrix-counts}
For $n\ge2$,
\begin{align}
 |\GL_n(q)|
 &=\prod_{j=0}^{n-1}(q^n-q^j)
 =q^{n^2}\prod_{i=1}^{n}(1-q^{-i}),                         \label{eq:GL-count}\\
 z_n(q)&>q^{n^2-1}.                                          \label{eq:z-lower}
\end{align}
Moreover, for fixed $n$ and $q\to\infty$,
\begin{equation}\label{eq:z-asymptotic}
  z_n(q)=q^{n^2-1}(1+O_n(q^{-1})).
\end{equation}
\end{lemma}

\begin{proof}
An invertible matrix is obtained by choosing its columns successively.  The first column can be any nonzero vector, giving $q^n-1$ choices.  Once $j$ linearly independent columns have been chosen, their span has $q^j$ vectors, so the next column has $q^n-q^j$ choices.  This proves the first product in \eqref{eq:GL-count}.  Factoring $q^n$ from every factor gives the second product. By \eqref{eq:z-definition} and \eqref{eq:GL-count}, we have
\[
 z_n(q)=q^{n^2}
 \left(1-\prod_{i=1}^{n}(1-q^{-i})\right).
\]
Because $n\ge2$, all factors are positive and
\[
  \prod_{i=1}^{n}(1-q^{-i})<1-q^{-1}.
\]
Therefore $z_n(q)>q^{n^2}q^{-1}=q^{n^2-1}$. Finally, for fixed $n$, expanding the finite product gives
\[
  \prod_{i=1}^{n}(1-q^{-i})=1-q^{-1}+O_n(q^{-2}).
\]
Substituting into the formula for $z_n(q)$ proves \eqref{eq:z-asymptotic}.
\end{proof}


We can now state the parity-sensitive bound in a compact form.

\begin{theorem}\label{thm:main}
For every prime power $q$ and every integer $n\ge2$,
\[
  q^n-1\leq\diss(\Gamma_n(q))\leq\diss(T_n(q))\leq U_n(q),
\]
where
\begin{equation}\label{eq:U-definition}
U_n(q)=
\begin{cases}
\displaystyle
\left\lfloor q^{n^2}
\frac{a_2(n,q)+2}{z_n(q)+a_2(n,q)}\right\rfloor,
&q\text{ even},\\[1.2em]
\displaystyle
\left\lfloor q^{n^2}
\frac{a_1(n,q)+2}{z_n(q)+a_1(n,q)}\right\rfloor,
&q\text{ odd}.
\end{cases}
\end{equation}
In addition,
\[
  \diss(T_n(q))\ge q^n.
\]
\end{theorem}

\begin{corollary}\label{cor:uniform}
For every prime power $q$ and every $n\ge2$,
\[
  \diss(\Gamma_n(q))\le\diss(T_n(q))
  \le q^{n^2-n+1}-1.
\]
\end{corollary}

\begin{corollary}\label{cor:asymptotic}
For fixed $n\ge2$ and $q\to\infty$ through prime powers,
\[
 U_n(q)=
 \begin{cases}
 (1+O_n(q^{-1}))q^{n^2-2n+2},&q\text{ even},\\
 (1+O_n(q^{-1}))q^{n^2-n+1},&q\text{ odd}.
 \end{cases}
\]
\end{corollary}

Theorem~\ref{thm:main} will be proved in two parts.  The lower bounds are
established in Section~\ref{sec:lower}.  The upper bounds require the
spectral preparation developed in Section~\ref{sec:spectral}.

\section{The extension-field lower bound}\label{sec:lower}

The regular representation of $\F_{q^n}$ produces a large set of matrices
whose pairwise sums have a particularly simple singularity condition.

\begin{proposition}\label{prop:lower}
For every prime power $q$ and every $n\ge2$,
\[
  \diss(\Gamma_n(q))\ge q^n-1
  \qquad\text{and}\qquad
  \diss(T_n(q))\ge q^n.
\]
More precisely, the constructed subgraph is independent when $q$ is even;
when $q$ is odd, it is a perfect matching together with one isolated vertex
in the total-graph case.
\end{proposition}

\begin{proof}
Let $K=\F_{q^n}$.  Since $K$ is an extension of $\F_q$ of degree $n$, it is
an $n$-dimensional vector space over $\F_q$.  Choose and fix an ordered
$\F_q$-basis $\mathcal B$ of $K$.

For each $x\in K$, define the $\F_q$-linear map
\[
  m_x:K\longrightarrow K,
  \qquad
  m_x(y)=xy.
\]
Let $\iota(x)$ be the matrix of $m_x$ with respect to $\mathcal B$.  The map
\[
  \iota:K\longrightarrow\Mat_n(q),
  \qquad x\longmapsto[m_x]_{\mathcal B}
\]
is an $\F_q$-algebra homomorphism.  Indeed, for all $x,y\in K$,
\[
  m_{x+y}=m_x+m_y,
  \qquad
  m_{xy}=m_x\circ m_y,
  \qquad
  m_1=\operatorname{id}_K.
\]
It is injective: if $\iota(x)=0$, then $m_x(1)=x=0$.

We next determine which matrices $\iota(x)$ are invertible.  If $x\ne0$,
then $m_x$ has inverse $m_{x^{-1}}$, so $\iota(x)\in\GL_n(q)$.  Conversely,
$m_0$ is the zero map and is not invertible.  Thus
\begin{equation}\label{eq:field-units}
  \iota(K^\times)\subseteq\GL_n(q),
  \qquad
  |\iota(K^\times)|=q^n-1.
\end{equation}

Let $x,y\in K^\times$ be distinct elements.  Additivity of $\iota$ gives
\[
  \iota(x)+\iota(y)=\iota(x+y).
\]
By the preceding invertibility criterion, this sum is singular if and only if $x+y=0$.  Therefore the induced graph on $\iota(K^\times)$ has the following form.

If $q$ is even, then $-x=x$.  The relation $x+y=0$ would imply $y=x$, which is excluded because the graph is simple and $x,y$ are distinct.  Hence $\iota(K^\times)$ is independent.

If $q$ is odd, then $x\ne-x$ for every $x\ne0$.  Each vertex $\iota(x)$ has exactly one neighbor inside the constructed set, namely $\iota(-x)$.  Thus the induced graph is a perfect matching.  In both parity cases it has maximum degree at most one, so \eqref{eq:field-units} gives
\[
  \diss(\Gamma_n(q))\ge q^n-1.
\]

For the total graph, include $\iota(0)$.  For every nonzero $x$, the sum
$\iota(0)+\iota(x)=\iota(x)$ is invertible, so $\iota(0)$ is isolated in the
subgraph induced by $\iota(K)$.  Consequently $\iota(K)$ is a dissociation
set in $T_n(q)$ of order $q^n$.
\end{proof}

\section{Spectral method and upper bounds}\label{sec:spectral}

\subsection{A Laplacian bound for \texorpdfstring{$k$}{k}-independent sets}
\label{sec:laplacian-bound}

Hoffman's ratio bound and its Laplacian versions are standard tools for
independent sets; see \cite{van-dam-haemers,godsil-newman,haemers}.  The
following elementary extension is the only general spectral inequality
needed here.

\begin{theorem}\label{thm:k-laplacian}
Let $G$ be a graph of order $v$, minimum degree $\delta$, and largest
Laplacian eigenvalue $\lambda_{\max}>0$.  For every integer $k\ge0$,
\begin{equation}\label{eq:k-laplacian-bound}
  \alpha_k(G)\leq
  v\left(1-\frac{\delta-k}{\lambda_{\max}}\right).
\end{equation}
In particular,
\begin{equation}\label{eq:diss-laplacian-bound}
  \diss(G)\leq
  v\left(1-\frac{\delta-1}{\lambda_{\max}}\right).
\end{equation}
\end{theorem}

\begin{proof}
Let $D\subseteq V(G)$ be a $k$-independent set and write $d=|D|$.  If
$d=0$, the assertion is immediate, so assume $d>0$.  Let $x\in\mathbb R^v$
be the indicator vector of $D$:
\[
  x_u=
  \begin{cases}
  1,&u\in D,\\
  0,&u\notin D.
  \end{cases}
\]
By the quadratic-form identity \eqref{eq:laplacian-basic}, an edge with both
endpoints in $D$ contributes $(1-1)^2=0$, an edge with both endpoints
outside $D$ contributes $(0-0)^2=0$, and an edge crossing from $D$ to its
complement contributes $1$.  Therefore
\begin{equation}\label{eq:indicator-cut}
  x^{\mathsf T}L_Gx=e_G(D,V(G)\setminus D).
\end{equation}

The degree sum over $D$ counts every edge of $G[D]$ twice and every crossing
edge once.  Hence
\begin{equation}\label{eq:degree-cut}
  e_G(D,V(G)\setminus D)
  =\sum_{u\in D}d_G(u)-2e_G(D).
\end{equation}
Since every vertex has degree at least $\delta$,
\[
  \sum_{u\in D}d_G(u)\ge\delta d.
\]
Since $\Delta(G[D])\le k$, the handshaking lemma applied to $G[D]$ gives
\[
  2e_G(D)=\sum_{u\in D}d_{G[D]}(u)\le kd.
\]
Combining these inequalities with \eqref{eq:indicator-cut} and
\eqref{eq:degree-cut}, we obtain the lower estimate
\begin{equation}\label{eq:cut-lower}
  x^{\mathsf T}L_Gx\ge(\delta-k)d.
\end{equation}

The vector $x$ need not be orthogonal to the kernel vector $\ind$, so we
center it.  Define
\[
  y=x-\frac{d}{v}\ind.
\]
Because $\ind^{\mathsf T}x=d$ and
$\ind^{\mathsf T}\ind=v$,
\[
  \ind^{\mathsf T}y=d-\frac{d}{v}v=0.
\]
Also, $L_G\ind=0$, and therefore
\begin{equation}\label{eq:centered-energy}
  y^{\mathsf T}L_Gy=x^{\mathsf T}L_Gx.
\end{equation}
A direct calculation gives
\begin{align}
 y^{\mathsf T}y
 &=\left(x-\frac{d}{v}\ind\right)^{\mathsf T}
   \left(x-\frac{d}{v}\ind\right)\notag\\
 &=x^{\mathsf T}x-\frac{2d}{v}\ind^{\mathsf T}x
   +\frac{d^2}{v^2}\ind^{\mathsf T}\ind\notag\\
 &=d-\frac{2d^2}{v}+\frac{d^2}{v}
 =d-\frac{d^2}{v}.                                         \label{eq:centered-norm}
\end{align}

The Rayleigh principle for the largest eigenvalue gives us
\[
  y^{\mathsf T}L_Gy\le\lambda_{\max}y^{\mathsf T}y.
\]
Using \eqref{eq:cut-lower}, \eqref{eq:centered-energy}, and
\eqref{eq:centered-norm}, we obtain
\[
  (\delta-k)d
  \le\lambda_{\max}\left(d-\frac{d^2}{v}\right).
\]
Dividing by $d>0$ and rearranging we get
\[
  \frac{\lambda_{\max}d}{v}
  \le\lambda_{\max}-\delta+k,
\]
and hence
\[
  d\le v\left(1-\frac{\delta-k}{\lambda_{\max}}\right).
\]
Finally, taking the maximum over all $k$-independent sets proves \eqref{eq:k-laplacian-bound} and the special case $k=1$ is \eqref{eq:diss-laplacian-bound}.
\end{proof}

\begin{remark}
For $k=0$, Theorem~\ref{thm:k-laplacian} becomes the standard Laplacian
upper bound
\[
  \alpha(G)\le v\left(1-\frac{\delta}{\lambda_{\max}}\right).
\]
For $k=1$, the only numerical change is the replacement of $\delta$ by
$\delta-1$, but this small change must be combined with a sufficiently
accurate value of $\lambda_{\max}$ to obtain the parity-sensitive results
of this paper.
\end{remark}

\subsection{Sum graphs and their Laplacian spectrum}\label{sec:sum-graphs}

Let $H$ be a finite abelian group written additively, and let $S\subseteq H$.
The \emph{sum graph} $\operatorname{Sum}(H,S)$ is the simple graph with
vertex set $H$ in which distinct $g,h$ are adjacent if and only if
$g+h\in S$.  This construction is closely related to an abelian Cayley
graph, but it uses a sum rather than a difference.  Adjacency spectra of
loopy versions of Cayley sum graphs have appeared in work such as
\cite{li-sum,devos}; the loopless Laplacian form used below is particularly
convenient for $T_n(q)$ and was emphasized by Nica \cite{nica}.

The graph need not be regular.  Indeed, for fixed $g\in H$, the map
$s\mapsto s-g$ sends $S$ bijectively to the set of all potential neighbors
of $g$.  The potential neighbor $s-g$ equals $g$ precisely when $s=2g$.
Therefore
\begin{equation}\label{eq:sum-degree}
  d(g)=|S|-\varepsilon_g,
  \qquad
  \varepsilon_g=
  \begin{cases}
  1,&2g\in S,\\
  0,&2g\notin S.
  \end{cases}
\end{equation}

Let $\widehat H$ be the character group of $H$.  Each character is a
homomorphism $\chi:H\to\mathbb C^\times$, and its values lie on the unit
circle.  In particular,
\[
  \chi(-g)=\chi(g)^{-1}=\overline{\chi(g)}.
\]
The characters form an orthogonal basis of the vector space of complex
functions on $H$.

\begin{theorem}
\label{thm:sum-spectrum}
Let $s=|S|$ and, for $\chi\in\widehat H$, define
\[
  \sigma_\chi=\sum_{a\in S}\chi(a).
\]
Every real character $\chi$ contributes the Laplacian eigenvalue
$s-\sigma_\chi$.  Every conjugate pair of nonreal characters
$\{\chi,\overline\chi\}$ contributes the two eigenvalues
\[
  s-|\sigma_\chi|
  \qquad\text{and}\qquad
  s+|\sigma_\chi|.
\]
Together these values, with multiplicity, form the full Laplacian spectrum
of $\operatorname{Sum}(H,S)$.
\end{theorem}

\begin{proof}
Let $L$ be the Laplacian operator of the sum graph and fix
$\chi\in\widehat H$.  By \eqref{eq:sum-degree}, the actual neighbor sum at
$g$ is
\[
  \sum_{h\sim g}\chi(h)
  =\sum_{a\in S}\chi(a-g)-\varepsilon_g\chi(g).
\]
The subtracted term removes the forbidden loop $h=g$ when $2g\in S$.  It
follows that
\begin{align*}
 (L\chi)(g)
 &=d(g)\chi(g)-\sum_{h\sim g}\chi(h)\\
 &=(s-\varepsilon_g)\chi(g)
   -\left(\sum_{a\in S}\chi(a-g)-\varepsilon_g\chi(g)\right)\\
 &=s\chi(g)-\sum_{a\in S}\chi(a-g).
\end{align*}
Using multiplicativity of the character,
\[
  \chi(a-g)=\chi(a)\chi(-g)
  =\chi(a)\overline{\chi(g)}.
\]
Consequently,
\begin{equation}\label{eq:L-character}
  L\chi=s\chi-\sigma_\chi\overline\chi.
\end{equation}

If $\chi$ is real, then $\overline\chi=\chi$ and $L\chi=(s-\sigma_\chi)\chi$.  Thus $\chi$ is an eigenvector with the stated eigenvalue. Suppose next that $\chi$ is nonreal. Since
\[
  \sigma_{\overline\chi}
  =\sum_{a\in S}\overline{\chi(a)}
  =\overline{\sigma_\chi},
\]
equation \eqref{eq:L-character} shows that the two-dimensional space
$\operatorname{span}\{\chi,\overline\chi\}$ is $L$-invariant.  With respect
to the ordered basis $(\chi,\overline\chi)$, the restriction of $L$ is
\[
  \begin{pmatrix}
  s&-\overline{\sigma_\chi}\\
  -\sigma_\chi&s
  \end{pmatrix}.
\]
Its characteristic polynomial is
\[
  (s-\lambda)^2-|\sigma_\chi|^2,
\]
so its eigenvalues are $s\pm|\sigma_\chi|$.

Finally, the character basis partitions into real characters and conjugate pairs of nonreal characters.  The one- and two-dimensional invariant spaces described above therefore form a direct sum of the whole function space and no additional Laplacian eigenvalues remain.
\end{proof}

\subsection{Characters of the matrix ring and a Gauss sum}
\label{sec:character-sum}

Fix a nontrivial additive character
$\psi:(\F_q,+)\to\mathbb C^\times$.  For $U\in\Mat_n(q)$, define
\begin{equation}\label{eq:matrix-character}
  \chi_U(X)=\psi(\Tr(U^{\mathsf T}X)).
\end{equation}
The bilinear form
\[
  \langle U,X\rangle=\Tr(U^{\mathsf T}X)
  =\sum_{i,j=1}^{n}u_{ij}x_{ij}
\]
is nondegenerate.  Hence $U\mapsto\chi_U$ is injective.  Both
$\Mat_n(q)$ and its character group have $q^{n^2}$ elements, so these are
all additive characters of $\Mat_n(q)$.  Also,
\begin{equation}\label{eq:character-conjugate}
  \overline{\chi_U}=\chi_{-U}.
\end{equation}

We now prove the matrix character sum used in the spectral calculation.  It
was evaluated by Li and Hu \cite{li-hu}; the proof below is included to make
the present argument self-contained.

\begin{theorem}\label{thm:matrix-gauss}
If $U\in\Mat_n(q)$ has rank $r\ge1$, then
\begin{equation}\label{eq:matrix-gauss}
  \sum_{A\in\GL_n(q)}\chi_U(A)
  =(-1)^r q^{n(n-1)/2}
   \prod_{i=1}^{n-r}(q^i-1)
  =(-1)^r a_r(n,q).
\end{equation}
\end{theorem}

\begin{proof}
Let $V=\F_q^n$, and $A$ be a matrix with the linear map $V\to V$ given by left multiplication on column vectors.  For a subspace
$W\le V$, let
\[
  \mathcal L_W=\{A\in\Mat_n(q):W\subseteq\ker A\}.
\]
If $k=\dim W$, then every row of a matrix in $\mathcal L_W$ is a linear functional vanishing on $W$, so every row lies in the annihilator $W^\perp$, which has dimension $n-k$.  The $n$ rows can be chosen independently, and therefore
\begin{equation}\label{eq:L-W-size}
  |\mathcal L_W|=q^{n(n-k)}.
\end{equation}
The M\"obius function of the lattice of subspaces of $V$ satisfies
\begin{equation}\label{eq:subspace-mobius}
  \mu(0,W)=(-1)^kq^{\binom{k}{2}}
  \qquad (k=\dim W).
\end{equation}
The M\"obius inversion on this lattice gives, for every linear map $A$,
\begin{equation}\label{eq:invertible-indicator}
  \ind_{\{\ker A=0\}}
  =\sum_{W\le\ker A}\mu(0,W).
\end{equation}
Indeed, the right-hand side is $1$ when $\ker A=0$ and is $0$ otherwise, which is the defining cancellation property of the M\"obius function. Now, using \eqref{eq:invertible-indicator} to expand the desired sum, we have
\begin{align}
 \sum_{A\in\GL_n(q)}\chi_U(A)
 &=\sum_{A\in\Mat_n(q)}
   \chi_U(A)\ind_{\{\ker A=0\}}\notag\\
 &=\sum_{A\in\Mat_n(q)}\chi_U(A)
   \sum_{W\le\ker A}\mu(0,W)\notag\\
 &=\sum_{W\le V}\mu(0,W)
   \sum_{A\in\mathcal L_W}\chi_U(A).                    \label{eq:mobius-expanded}
\end{align}

The inner sum is a character sum over the additive subspace $\mathcal L_W$.  Such a sum is $0$ unless the character is trivial on that
subspace, in which case it equals $|\mathcal L_W|$.  We now characterize triviality.  Under the entry wise matrix pairing, the rows of
$\mathcal L_W$ range independently over $W^\perp$.  Therefore
\[
  \langle U,A\rangle=0\quad\text{for every }A\in\mathcal L_W
\]
if and only if every row of $U$ belongs to $(W^\perp)^\perp=W$.  Equivalently,
\begin{equation}\label{eq:row-condition}
  \chi_U|_{\mathcal L_W}\text{ is trivial}
  \quad\Longleftrightarrow\quad
  \Row(U)\subseteq W.
\end{equation}

Let $R=\Row(U)$, so $\dim R=r$.  Substituting \eqref{eq:L-W-size}, \eqref{eq:subspace-mobius} and \eqref{eq:row-condition} into \eqref{eq:mobius-expanded} gives
\begin{equation}\label{eq:sum-over-W}
 \sum_{A\in\GL_n(q)}\chi_U(A)
 =\sum_{W\,\supseteq\,R}
  (-1)^{\dim W}q^{\binom{\dim W}{2}}q^{n(n-\dim W)}.
\end{equation}

The number of $k$-dimensional subspaces $W$ containing the fixed $r$-dimensional space $R$ is the Gaussian binomial coefficient
$\qbinom{n-r}{k-r}$.  Put $m=n-r$ and write $k=r+j$.  Then \eqref{eq:sum-over-W} becomes
\begin{align}
 &(-1)^r q^{\binom r2+n(n-r)}
 \sum_{j=0}^{m}
 \qbinom{m}{j}(-1)^j q^{\binom j2}(q^{r-n})^j.             \label{eq:q-binomial-sum}
\end{align}
The finite $q$-binomial theorem states that
\begin{equation}\label{eq:q-binomial-theorem}
  \sum_{j=0}^{m}\qbinom{m}{j}(-1)^j q^{\binom j2}x^j
  =\prod_{i=0}^{m-1}(1-xq^i).
\end{equation}
Taking $x=q^{r-n}=q^{-m}$ in \eqref{eq:q-binomial-theorem}, the sum in \eqref{eq:q-binomial-sum} equals
\[
  \prod_{i=0}^{m-1}(1-q^{-m+i})
  =\prod_{h=1}^{m}(1-q^{-h}).
\]
Consequently,
\begin{align*}
 \sum_{A\in\GL_n(q)}\chi_U(A)
 &=(-1)^r q^{\binom r2+n(n-r)}
   \prod_{h=1}^{n-r}(1-q^{-h})\\
 &=(-1)^r q^{\binom r2+n(n-r)-\binom{n-r+1}{2}}
   \prod_{h=1}^{n-r}(q^h-1).
\end{align*}
A direct simplification of the exponent gives
\[
  \binom r2+n(n-r)-\binom{n-r+1}{2}
  =\frac{n(n-1)}{2},
\]
which matches with \eqref{eq:matrix-gauss}.
\end{proof}

\begin{corollary}\label{cor:singular-character-sum}
If $U\ne0$ has rank $r$, then
\begin{equation}\label{eq:singular-character-sum}
  \sum_{S\in\cS_n(q)}\chi_U(S)
  =-(-1)^r a_r(n,q).
\end{equation}
\end{corollary}

\begin{proof}
Since $U\ne0$, the character $\chi_U$ is nontrivial.  Character orthogonality on the additive group $\Mat_n(q)$ gives
\[
  \sum_{X\in\Mat_n(q)}\chi_U(X)=0.
\]
The matrix ring is the disjoint union of $\GL_n(q)$ and $\cS_n(q)$, so
\[
  \sum_{S\in\cS_n(q)}\chi_U(S)
  =-\sum_{A\in\GL_n(q)}\chi_U(A).
\]
Finally, we apply Theorem~\ref{thm:matrix-gauss} to get the result.
\end{proof}

\subsection{The total graph and the upper bounds}\label{sec:upper}

The total graph is the sum graph
\begin{equation}\label{eq:T-as-sum}
  T_n(q)=\operatorname{Sum}(\Mat_n(q),\cS_n(q)).
\end{equation}
We first determine its minimum degree and largest Laplacian eigenvalue.

\begin{lemma}\label{lem:degree}
The graph $T_n(q)$ has order $q^{n^2}$ and minimum degree
\[
  \delta(T_n(q))=z_n(q)-1.
\]
More precisely, if $q$ is even then every vertex has degree $z_n(q)-1$; if $q$ is odd, singular vertices have degree $z_n(q)-1$ and invertible vertices have degree $z_n(q)$.
\end{lemma}

\begin{proof}
The order is $|\Mat_n(q)|=q^{n^2}$.  Fix $A\in\Mat_n(q)$.  For each $S\in\cS_n(q)$ there is a unique matrix
\[
  B=S-A
\]
with $A+B=S$ singular.  Thus there are $z_n(q)$ potential neighbors.  The only possible forbidden one is $B=A$, because the graph has no loops.  This occurs exactly when $S=2A$ is singular.  Therefore
\[
  d_{T_n(q)}(A)=
  \begin{cases}
  z_n(q)-1,&2A\text{ is singular},\\
  z_n(q),&2A\text{ is invertible}.
  \end{cases}
\]

If $q$ is even, then $2A=0$ is singular for every $A$, so all degrees equal $z_n(q)-1$.  If $q$ is odd, multiplication by the nonzero scalar $2$ preserves rank, and hence $2A$ is singular exactly when $A$ is singular. Thus the asserted degree description and minimum degree follow.
\end{proof}

\begin{theorem}\label{thm:T-spectrum}
Apart from the eigenvalue $0$, the distinct Laplacian eigenvalues of
$T_n(q)$ are
\begin{equation}\label{eq:spectrum-even}
  z_n(q)+(-1)^r a_r(n,q),
  \qquad 1\le r\le n,
\end{equation}
when $q$ is even, and
\begin{equation}\label{eq:spectrum-odd}
  z_n(q)\pm a_r(n,q),
  \qquad 1\le r\le n,
\end{equation}
when $q$ is odd.  Consequently,
\begin{equation}\label{eq:lambda-max}
  \lambda_{\max}(T_n(q))=
  \begin{cases}
  z_n(q)+a_2(n,q),&q\text{ even},\\
  z_n(q)+a_1(n,q),&q\text{ odd}.
  \end{cases}
\end{equation}
\end{theorem}

\begin{proof}
We apply Theorem~\ref{thm:sum-spectrum} to the sum-graph description \eqref{eq:T-as-sum}.  Here $s=|\cS_n(q)|=z_n(q)$.  The trivial character $\chi_0$ is real and has
\[
  \sigma_{\chi_0}=z_n(q),
\]
so it contributes the expected eigenvalue $z_n(q)-z_n(q)=0$.

Now let $U\ne0$ have rank $r$.  Corollary \ref{cor:singular-character-sum} gives
\begin{equation}\label{eq:sigma-U}
  \sigma_{\chi_U}=-(-1)^r a_r(n,q).
\end{equation}

Suppose first that $q$ is even.  The additive group of $\Mat_n(q)$ has exponent two: $X+X=0$ for every $X$.  Hence every character satisfies
\[
  \chi_U(X)^2=\chi_U(2X)=1,
\]
so every character is real-valued.  The real-character part of
Theorem~\ref{thm:sum-spectrum} and \eqref{eq:sigma-U} yield
\[
  z_n(q)-\sigma_{\chi_U}
  =z_n(q)+(-1)^r a_r(n,q),
\]
which proves \eqref{eq:spectrum-even}.

Suppose next that $q$ is odd.  The additive group $\Mat_n(q)$ has odd order.  A real character takes values in $\mathbb R\cap\{z\in\mathbb C:|z|=1\}=\{1,-1\}$.  Its image is therefore a group of order at most two.  Since the domain has odd order, the image must be trivial.  Thus the only real character is $\chi_0$, and every nontrivial character belongs to a nonreal conjugate pair.  By Theorem~\ref{thm:sum-spectrum} and \eqref{eq:sigma-U}, that pair contributes
\[
  z_n(q)\pm|\sigma_{\chi_U}|
  =z_n(q)\pm a_r(n,q),
\]
which proves \eqref{eq:spectrum-odd}.

It remains to identify the largest value.  From \eqref{eq:a-definition}, for $1\le r<n$,
\begin{equation}\label{eq:a-ratio}
  \frac{a_r(n,q)}{a_{r+1}(n,q)}=q^{n-r}-1\ge1.
\end{equation}
Thus $a_1\ge a_2\ge\cdots\ge a_n>0$.

For odd $q$, the largest value in \eqref{eq:spectrum-odd} is plainly $z_n(q)+a_1(n,q)$.  For even $q$, a positive addition in \eqref{eq:spectrum-even} occurs exactly when $r$ is even.  Since $n\ge2$, the smallest permitted even rank is $r=2$, and monotonicity
\eqref{eq:a-ratio} shows that $a_2$ is the largest positive contribution. Hence the largest eigenvalue is $z_n(q)+a_2(n,q)$.  This proves \eqref{eq:lambda-max}.
\end{proof}

We can now prove the upper half of the main theorem.

\begin{proof}[Proof of Theorem~\ref{thm:main}]
Proposition~\ref{prop:lower} gives
\[
  \diss(\Gamma_n(q))\ge q^n-1
  \qquad\text{and}\qquad
  \diss(T_n(q))\ge q^n.
\]
For the upper bound on the total graph, we apply Theorem~\ref{thm:k-laplacian} with $G=T_n(q)$ and $k=1$.  By Lemma~\ref{lem:degree}, the graph has order $q^{n^2}$ and minimum degree $z_n(q)-1$.  Therefore
\begin{align}
 \diss(T_n(q))
 &\le q^{n^2}\left(
 1-\frac{(z_n(q)-1)-1}{\lambda_{\max}(T_n(q))}
 \right)\notag\\
 &=q^{n^2}\left(
 1-\frac{z_n(q)-2}{\lambda_{\max}(T_n(q))}
 \right).                                                   \label{eq:T-general-upper}
\end{align}
If $q$ is even, Theorem~\ref{thm:T-spectrum} gives $\lambda_{\max}=z_n(q)+a_2(n,q)$.  Substituting into \eqref{eq:T-general-upper} yields
\begin{align*}
 \diss(T_n(q))
 &\le q^{n^2}\left(
 1-\frac{z_n(q)-2}{z_n(q)+a_2(n,q)}
 \right)\\
 &=q^{n^2}
 \frac{a_2(n,q)+2}{z_n(q)+a_2(n,q)}.
\end{align*}
Because the left-hand side is an integer, it is at most the floor of the right-hand side.  This is the even case of $U_n(q)$. If $q$ is odd, the same calculation with
$\lambda_{\max}=z_n(q)+a_1(n,q)$ gives
\[
 \diss(T_n(q))
 \le q^{n^2}
 \frac{a_1(n,q)+2}{z_n(q)+a_1(n,q)},
\]
and taking the floor gives the odd case of $U_n(q)$.

Finally, \eqref{eq:induced-subgraph} shows that every dissociation set in $\Gamma_n(q)$ induces exactly the same edges when viewed as a vertex set of $T_n(q)$.  Hence it is also a dissociation set in $T_n(q)$, and
\[
  \diss(\Gamma_n(q))\le\diss(T_n(q)).
\]
All assertions of Theorem~\ref{thm:main} now follow.
\end{proof}

\begin{proof}[Proof of Corollary~\ref{cor:uniform}]
Put
\[
  N=n^2,
  \qquad
  Q=q^{N-n}=q^{n^2-n}.
\]
We first obtain a parity-free upper estimate for the largest eigenvalue. From \eqref{eq:a-definition},
\begin{align*}
 a_1(n,q)
 &=q^{n(n-1)/2}\prod_{i=1}^{n-1}(q^i-1)\\
 &<q^{n(n-1)/2}\prod_{i=1}^{n-1}q^i\\
 &=q^{n(n-1)/2+n(n-1)/2}\\
 &=q^{n^2-n}=Q.
\end{align*}
Also $a_2(n,q)\le a_1(n,q)$.  In either parity case, Theorem~\ref{thm:T-spectrum} therefore implies
\begin{equation}\label{eq:lambda-uniform}
  \lambda_{\max}(T_n(q))<z_n(q)+Q.
\end{equation}

Since $z_n(q)>q^{N-1}\ge8$, the number $z_n(q)-2$ is positive.  Starting from \eqref{eq:T-general-upper} and using
\eqref{eq:lambda-uniform}, we get
\begin{align}
 \frac{\diss(T_n(q))}{q^N}
 &\le1-\frac{z_n(q)-2}{\lambda_{\max}(T_n(q))}\notag\\
 &<1-\frac{z_n(q)-2}{z_n(q)+Q}\notag\\
 &=\frac{Q+2}{z_n(q)+Q}.                                 \label{eq:uniform-step-one}
\end{align}
Using Lemma~\ref{lem:matrix-counts} we have $z_n(q)>q^{N-1}$, so
\begin{equation}\label{eq:uniform-step-two}
  \frac{Q+2}{z_n(q)+Q}
  <\frac{Q+2}{q^{N-1}+Q}.
\end{equation}

We claim that the final fraction is at most $q^{-(n-1)}$.  All quantities are positive, so this is equivalent to
\[
  q^{n-1}(Q+2)\le q^{N-1}+Q.
\]
Now
\[
  q^{n-1}Q=q^{n-1}q^{N-n}=q^{N-1},
\]
and
\[
  \frac{Q}{q^{n-1}}
  =q^{N-2n+1}=q^{(n-1)^2}\ge2.
\]
Thus $2q^{n-1}\le Q$, proving the claim.  Combining \eqref{eq:uniform-step-one} and \eqref{eq:uniform-step-two},
\[
  \frac{\diss(T_n(q))}{q^N}<q^{-(n-1)}.
\]
Multiplication by $q^N$ gives
\[
  \diss(T_n(q))<q^{N-n+1}.
\]
Since the dissociation number is an integer,
\[
  \diss(T_n(q))\le q^{N-n+1}-1.
\]
The inequality for $\Gamma_n(q)$ follows from $\diss(\Gamma_n(q))\le\diss(T_n(q))$.
\end{proof}

\begin{proof}[Proof of Corollary~\ref{cor:asymptotic}]
Lemma~\ref{lem:matrix-counts} gives
\[
  z_n(q)=q^{n^2-1}(1+O_n(q^{-1})).
\]
For fixed $n$, factor the highest power of $q$ from the products in
\eqref{eq:a-definition}.  For $r=1$,
\begin{align*}
 a_1(n,q)
 &=q^{n(n-1)/2}
   q^{1+2+\cdots+(n-1)}
   \prod_{i=1}^{n-1}(1-q^{-i})\\
 &=q^{n^2-n}(1+O_n(q^{-1})).
\end{align*}
For $r=2$,
\begin{align*}
 a_2(n,q)
 &=q^{n(n-1)/2}
   q^{1+2+\cdots+(n-2)}
   \prod_{i=1}^{n-2}(1-q^{-i})\\
 &=q^{(n-1)^2}(1+O_n(q^{-1})).
\end{align*}
The second formula remains valid for $n=2$, when the product is empty.

For odd $q$, substitute the asymptotics for $z_n(q)$ and $a_1(n,q)$ into
the second line of \eqref{eq:U-definition}.  Since $a_1(n,q)=o(z_n(q))$
and $2=o(a_1(n,q))$,
\[
 q^{n^2}\frac{a_1(n,q)+2}{z_n(q)+a_1(n,q)}
 =(1+O_n(q^{-1}))q^{n^2-n+1}.
\]
For even $q$, the same argument with $a_2(n,q)$ gives
\[
 q^{n^2}\frac{a_2(n,q)+2}{z_n(q)+a_2(n,q)}
 =(1+O_n(q^{-1}))q^{(n-1)^2+1}
 =(1+O_n(q^{-1}))q^{n^2-2n+2}.
\]
Taking a floor changes either quantity by less than one and therefore does
not affect its asymptotic form.
\end{proof}

\section{Examples and concluding remarks}\label{sec:examples}

The next table lists the field-construction lower bound, the number of
singular matrices, the largest Laplacian eigenvalue, and the explicit upper
bound for several small parameters.  Each entry follows directly from
\eqref{eq:GL-count}, \eqref{eq:a-definition}, \eqref{eq:lambda-max}, and
\eqref{eq:U-definition}.

\begin{center}
\begin{tabular}{ccrrrr}
\toprule
$n$&$q$&$q^n-1$&$z_n(q)$&$\lambda_{\max}(T_n(q))$&$U_n(q)$\\
\midrule
2&2&3&10&12&5\\
2&3&8&33&39&16\\
2&4&15&76&80&19\\
3&2&7&344&352&14\\
\bottomrule
\end{tabular}
\end{center}

The spectral bound is not always exact.  In the smallest case the regular
graph itself can be recognized explicitly.

\begin{proposition}\label{prop:small}
The graph $\Gamma_2(2)$ is isomorphic to $K_{3,3}$.  Consequently,
\[
  \diss(\Gamma_2(2))=3.
\]
\end{proposition}

\begin{proof}
The group $\GL_2(2)$ has order
\[
  (2^2-1)(2^2-2)=3\cdot2=6.
\]
It acts on the three nonzero vectors of $\F_2^2$.  This action is faithful:
an invertible linear map fixing all three nonzero vectors fixes a basis and
is therefore the identity.  Hence the action identifies
$\GL_2(2)$ with the symmetric group $S_3$.

For distinct $A,B\in\GL_2(2)$, factor
\[
  A+B=A(I+A^{-1}B).
\]
Since $A$ is invertible,
\[
  \det(A+B)=0
  \quad\Longleftrightarrow\quad
  \det(I+A^{-1}B)=0.
\]
Put $g=A^{-1}B$.  Because $A\ne B$, we have $g\ne I$.  The matrix $I+g$
is singular in characteristic two exactly when $g$ has eigenvalue $1$.

There are three elements of order two and two elements of order three in
$\GL_2(2)\cong S_3$.  If $g$ has order two, then
\[
  g^2-I=(g-I)^2=0
\]
in characteristic two, so $1$ is an eigenvalue.  If $g$ has order three,
then its minimal polynomial is $x^2+x+1$, which has no root $1$ over
$\F_2$.  Thus the connection set consists exactly of the three elements of
order two, corresponding under the isomorphism with $S_3$ to the three
transpositions.

It follows that $\Gamma_2(2)$ is the Cayley graph of $S_3$ generated by all
transpositions.  Multiplication by a transposition changes permutation
parity, so every edge joins an even permutation to an odd permutation.
Conversely, if $\sigma$ is even and $\tau$ is odd, then
$\sigma^{-1}\tau$ is an odd element of $S_3$, hence one of the three
transpositions.  Therefore every even permutation is adjacent to every odd
permutation.  There are three of each, proving
\[
  \Gamma_2(2)\cong K_{3,3}.
\]

Finally, let a dissociation set in $K_{3,3}$ contain $a$ vertices from one
part and $b$ from the other.  If $a,b>0$, every selected vertex in the first
part has induced degree $b$ and every selected vertex in the second part has
induced degree $a$.  Maximum degree at most one then forces $a\le1$ and
$b\le1$, so the set has order at most two.  If it meets only one part, it
has order at most three, and a full part is an independent set of order
three.  Hence $\diss(K_{3,3})=3$.
\end{proof}

\subsection{Concluding remarks}

In this paper, we studied the dissociation number of the regular graph
$\Gamma_n(q)$ associated with the matrix ring $\Mat_n(\F_q)$.  By embedding
$\F_{q^n}$ into $\Mat_n(\F_q)$ through its regular representation, we
constructed a dissociation set of order $q^n-1$ in $\Gamma_n(q)$ and one of
order $q^n$ in the total graph $T_n(q)$.  The induced graph arising from
this construction is independent in even characteristic and is a matching
in odd characteristic.

For the upper bounds, we proved a general Laplacian inequality for
$k$-independent sets and specialized it to dissociation sets.  We then
described the Laplacian spectrum of finite abelian sum graphs and evaluated
the required character sums over $\Mat_n(\F_q)$.  Applying these results to
$T_n(q)$ produced an explicit parity-sensitive bound and hence the
two-sided estimate
\[
  q^n-1\le\diss(\Gamma_n(q))\le U_n(q),
\]
as well as the uniform estimate
\[
  \diss(\Gamma_n(q))\le q^{n^2-n+1}-1.
\]
For fixed $n$, the explicit upper bound has order
$q^{n^2-2n+2}$ when $q$ is even and order $q^{n^2-n+1}$ when $q$ is odd.
This distinction follows from the rank of the character that produces the
largest Laplacian eigenvalue: rank two in even characteristic and rank one
in odd characteristic.  Finally, the small-parameter analysis gives the
exact value $\diss(\Gamma_2(2))=3$.

\section*{Declaration on the use of AI}
The author used generative AI tools to assist in discussing proof strategies, checking proofs, and improving the exposition. The author takes full responsibility for the mathematical arguments, results, and conclusions, all of which were carefully reviewed and verified by the author.

\section*{Declaration of competing interest}

The author declares no known competing financial interests or personal relationships that could have appeared to influence the work reported in this article.


\section*{Data availability}

Data sharing is not applicable to this article as no datasets were generated or analyzed during the current study.

\end{document}